\documentclass[preprint,12pt]{elsarticle}

\usepackage{array}
\usepackage{multirow}
\usepackage{xcolor}
\usepackage[breaklinks, citecolor=magenta,colorlinks=true]{hyperref}
\usepackage{amssymb}
\usepackage{lscape}

\usepackage{graphicx}
\usepackage{amssymb}
\usepackage{bbm}
\usepackage{bm}
\usepackage{amsmath}
\usepackage{mathtools}
\usepackage{tabularx, booktabs, makecell, caption} 
\usepackage{dcolumn}

\usepackage{url}
\usepackage{amsthm}
\newtheorem{theorem}{Theorem}[section]
\newtheorem{lemma}[theorem]{Lemma}

\newtheorem{definition}{Definition}[section] 
\newtheorem{question}{Question}[section]

\usepackage{breakurl}

\makeatletter
\renewcommand*\env@matrix[1][*\c@MaxMatrixCols c]{%
  \hskip -\arraycolsep
  \let\@ifnextchar\new@ifnextchar
  \array{#1}}
\makeatother

\begin{document}

\begin{frontmatter}



\title{Rational Distance Sets on a Parabola Using Pythagorean Triplets}


\author[inst1]{Sayak Bhattacharjee}
\ead{sayakb@iitk.ac.in}

\affiliation[inst1]{organization={Department of Physics},
            addressline={Indian Institute of Technology Kanpur}, 
            city={Kanpur},
            postcode={208016}, 
            state={Uttar Pradesh},
            country={India}}

\author[inst1]{Divyam Jain}

\begin{abstract}
We study \(N\)-point rational distance sets ($\textrm{\textrm{RDS}}(N)$) on the parabola \(y=x^2\). Previous approaches to the problem include efforts made using elliptic curves and diophantine chains, with successful analysis for $N\leq 4$. We extend the analysis for arbitrary $N$ by establishing a correspondence between $\textrm{\textrm{RDS}}(N)$s and Pythagorean triplets. Our main result gives sufficient and necessary conditions for the existence and nature of the $\textrm{\textrm{RDS}}(N)$s for arbitrary \(N\). Our approach also leads to an efficient computational algorithm to construct new $\textrm{\textrm{RDS}}(N)$s, and we provide multiple new examples of $\textrm{\textrm{RDS}}(N)$s for four and five points. The correspondence with Pythagorean triplets also helps to study the density of the solutions and we reproduce density results for \(N=2\) and \(3\).        
\end{abstract}




\end{frontmatter}


\section{Introduction}\label{sec:intro}
 
We define a rational distance set as follows.

\begin{sloppypar}
 \begin{definition}[Rational distance set]
A rational distance set on the parabola $y=x^2$, denoted by $\textrm{RDS}(N)$, is a set of $N$ points with rational coordinates such that all of the pairwise distances are rational.  
\end{definition}   
\end{sloppypar}

In 2000, Dean asks Campbell \cite{G_Campbell} the following question: Is it possible to find a rational distance set with four non-concyclic points on the parabola \(y = x^2\)? An elementary geometric proof suggests that infinitely many 3-point rational distance sets exist \cite{N_Dean}. Campbell \cite{G_Campbell} extends this to 4, and provides a 5-point example, albeit with 4 concyclic points, using elliptic curve analysis. A parametrization using diophantine chains \cite{A_Choudhry} shows that infinitely many $\textrm{\textrm{RDS}}(4)$s exist on \(y=x^2\), and also provides families of almost perfect solutions for larger \(N\). A natural question then arises, namely,

\begin{question}
If finite, what is the maximum number of points $N$ that can constitute an $\textrm{\textrm{RDS}}(N)$ on \(y = x^2\)? 
\end{question}\label{central_ques}

The first part of this question has been proven affirmatively. In 1960, Ulam and Erd\H{o}s \cite{S_Ulam,P_Erdos} conjectured that there exists no everywhere dense rational distance set in the plane. In 2010, Solymosi and Zeeuw \cite{Solymosi_Zeeuw} prove this (unconditionally) for algebraic curves, showing that no irreducible algebraic curve other than a line or a circle contains an infinite rational distance set. This implies that the maximum $N$ in Question \ref{central_ques} is indeed finite. Recently, conditional proofs of the Erd\H{o}s-Ulam conjecture using the Bombieri-Lang conjecture \cite{J_Shaffaf, T_Tao}, and using the abc conjecture \cite{H_Pasten} have been constructed. It has also been recently conditionally shown that exists a uniform bound on the maximum $N$ that is independent of the actual $\textrm{\textrm{RDS}}(N)$ in question \cite{ascher}. 

Given the above bound, a natural tendency is to attempt to access examples of $\textrm{\textrm{RDS}}(N)$ with large cardinalities. In fact, an $\textrm{\textrm{RDS}}(6)$ is still unidentified on $y=x^2$. Our work in this article suggests a scheme that precisely enables the above. The central result of this article is Theorem \ref{the:poidis}, which establishes that the nature and existence of the $\textrm{\textrm{RDS}}(N)$s for arbitrary $N$, giving explicit expressions for the coordinates of the $\textrm{\textrm{RDS}}(N)$s and the conditions for their existence in terms of ratios of the non-hypotenuse lengths of a Pythagorean triplet, conveniently referred to in this article as a Pythagorean ratio. Clearly, a rational distance set can admit coordinates in both the rationals and irrationals, however, we restrict to rational distance sets constructed via rational points exclusively, without any loss of generality. Throughout the article, the word \textit{`triplet'} is used only to refer to Pythagorean triplets, while the word \textit{`N-tuple'} is used to refer to the $x$-coordinates of an $\textrm{\textrm{RDS}}(N)$; in particular for \(N=3\), the word \textit{`triple'} has been used.



\begin{figure}
\centering
\includegraphics[width=\linewidth/2]{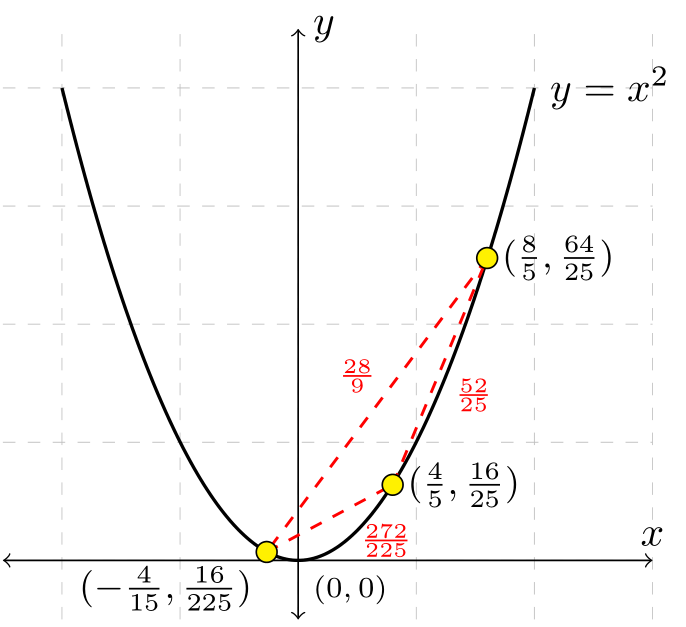}
\caption{ A three-point rational point-distance set on the parabola \(y=x^2\). Here, the solution set \(\mathbf{x_3}=(\begin{smallmatrix}
-\frac{4}{15}&\frac{8}{5}&\frac{4}{5}\end{smallmatrix})^T\) is obtained from the Pythagorean ratio set \(\bm{\psi_3}=(\begin{smallmatrix}
\frac{4}{3}&\frac{8}{15}&\frac{12}{5}\end{smallmatrix})^T\). The red numbers denote the rational distances between each pair of points (yellow circles).}
\label{fig:parabola}
\end{figure}

The rest of the paper is arranged as follows. In Sec.~\ref{prelimiaries}, we discuss the preliminaries required to set up the problem. We provide the main result of this work in Sec.~\ref{main_section}, along with computationally obtained examples. Subsequently, in Sec.~\ref{density_section}, we show an application of our correspondence by demonstrating the density of $N=2$ and $3$ $\textrm
{RDS}(N)$s, with an outlook for future work provided in Sec.~\ref{outlook}. 

\section{Preliminaries}\label{prelimiaries}

We establish a correspondence between the set of Pythagorean triplets and an \textrm{\textrm{RDS}}($N$). We formally define a Pythagorean triplet and the more relevant quantity, a Pythagorean ratio, which we use in our study. 

\begin{definition}\label{Pyth_def}[Pythagorean Triplet, Ratio] A Pythagorean triplet \((\alpha,\beta,\gamma)\) is an ordered 3-tuple such that $\alpha,\beta,\gamma \in \mathbb{Z}\backslash \{0\}$ with \(\alpha^2+\beta^2=\gamma^2\). Without loss of generality, we restrict \(\gamma>0\). Further, a triplet is,
\begin{enumerate}
\item primitive, if \(\alpha,\beta\) and \(\gamma\) are pairwise coprime.
\item all-positive, if \(\alpha\) and \(\beta\) are positive.
\item positive, if \(\alpha\beta>0\). \item negative, if \(\alpha\beta<0\).
\item naturally-ordered, if \(|\alpha|<|\beta|\:(<\gamma)\).
\item oppositely-ordered, if \(|\beta|<|\alpha|\:(<\gamma)\). 
\end{enumerate}

For each Pythagorean triplet \((\alpha,\beta,\gamma)\), we define a Pythagorean ratio, \(\psi:=\frac{\beta}{\alpha}\), as the ratio of the non-hypotenuse lengths of the Pythagorean triplet. Further, we also define a zero Pythagorean ratio, i.e. \(\psi_0:=0 \).
\end{definition}

In line with the parametrizations considered by Campbell and Chowdhury, we observe the following. 

\begin{lemma} \label{sum_psi}
Points \(P_i=(x_i,y_i)\) and \(P_j=(x_j,y_j)\) belong to an \textrm{\textrm{RDS}}$(N)$ if and only if \(x_i+x_j=\psi_{ij}\) where \(\psi_{ij}\) is a Pythagorean ratio chosen apriori.
\end{lemma}
\begin{proof}
The distance between $P_i$ and $P_j$ is given by  $((x_j-x_i)^2+(y_j-y_i)^2)^\frac{1}{2}=|x_j-x_i|(1+(x_i+x_j)^2)^\frac{1}{2}$. Since, \(x_i\) and \(x_j\) are rational, \(|x_j-x_i|\) is rational, and hence, we only need \((1+(x_i+x_j)^2)^\frac{1}{2}\) to be rational for the distance to be rational. Choose a Pythagorean triplet \((\alpha_{ij}, \beta_{ij}, \gamma_{ij})\). Noting that  \(\alpha_{ij}^2+\beta_{ij}^2=\gamma_{ij}^2\), define \((1+(x_i+x_j)^2)^\frac{1}{2}:=\frac{\gamma_{ij}}{\alpha_{ij}}\). We thus obtain:

\begin{equation}
x_i+x_j=\frac{\beta_{ij}}{\alpha_{ij}}=\psi_{ij}   \label{basic} 
\end{equation}

Observe, that the holds trivially when \(\psi_{ij}=\psi_0\). The opposite direction is also easy to show.
\end{proof}

As described previously, our objective is to perform an analysis for general $N$. We thus notice that the number of pairwise distances for an $\textrm{\textrm{RDS}}(N)$ is $N\choose 2$. Thus, we must solve a system of $N\choose 2$ equations of the form of Equation \ref{basic}, to get the $x$-coordinates of the $\textrm{\textrm{RDS}}(N)$. These can be compactly written in the form of a matrix equation. To do so, we need to define a coefficient matrix, which we define as follows. 

\begin{definition}[\(N_{(2)}\)-Indices Set]
A set \(N_{(2)}\)-Indices is the ordered set of all 2-combinations (without repetition) of the first \(N\) natural numbers with elements arranged in lexicographical order. There are \(N\choose 2\) elements in the set.   
\end{definition}

\begin{sloppypar}
As an example, the $4_{(2)}$-Indices set is the ordered set $\{(1,2), (1,3), (1,4), (2,3), (2,4), (3,4)\}$. Such a set may be used to define the coefficient matrix.
\end{sloppypar}

\begin{sloppypar}
\begin{definition} (Coefficient Matrix)
Define a coefficient matrix \(\mathbf{C}_N\), corresponding to the \(N_{(2)}\)-Indices set, of size \({N\choose 2}\times N\) so that if the \(i^{th}\) element of the set is (\(m_i,n_i\)),

\begin{equation*}
(\mathbf{C}_{N})_{ij}:=\begin{cases}
1 & j=m_i,j=n_i\\  0 & j\neq m_i,j\neq n_i\end{cases}
\end{equation*}
\end{definition}
\end{sloppypar}

The system of equation for \(N\) points explicitly is
\begin{equation}
    x_{m_i}+x_{n_i}=\psi_{m_in_i} \hspace{0.5cm} \left(1\leq i\leq{N\choose 2}\right) \label{sys_eq}
\end{equation}
where $(m_i,n_i)$ is the $i^{th}$ element of the $N_{(2)}$-Indices set.  

We next define two column vectors \(\mathbf{x}_{N}\) of size \(N\times1\) and \(\bm{\psi}_N\) of size \({N\choose 2}\times1\) such that \(\mathbf{x}_{N}:=(\begin{smallmatrix}x_1&x_2&\cdots&x_N\end{smallmatrix})^T\) and \(\bm{\psi}_N:=(\begin{smallmatrix}
    \psi_{m_1n_1}&\psi_{m_2n_2}&\cdots&\psi_{m_{N\choose 2}n_{N\choose 2}}\end{smallmatrix})^T\). 
Thus, the system of equations in Equation \ref{sys_eq} is equivalently written in matrix form as
\begin{equation}
    \mathbf{C}_N\mathbf{x}_{N}=\bm{\psi}_N \label{mat_eq}
\end{equation}

Let $\mathbf{A}[a, \hdots, b;]\equiv \mathbf{A}^{[a:b]}$ $(a<b)$ denote the submatrix given by the $a$ to $b$ (both inclusive) rows of the matrix $\mathbf{A}$. Then, we notice that any solution $\mathbf{x}_N$ of Equation \ref{mat_eq} also satisfies 

\begin{equation}
    \mathbf{C}_N^{[1:N]}\mathbf{x}_{N}=\bm{\psi}_N^{[1:N]} \label{mat_eq1}.
\end{equation}

This observation is key to most of the analysis performed in the article, and will be used and discussed in greater detail soon. We first present a few useful results related to the coefficients matrix. 

\begin{lemma} \label{lem:rankofC}
 Rank(\(\mathbf{C}_2)=1\). For \(N\geq3\),  rank(\(\mathbf{C}_N) = N\). 
\end{lemma}
\begin{proof}
\(\mathbf{C}_2=\left(\begin{smallmatrix} 1 & 1
 \end{smallmatrix}\right)\), and thus, rank(\(\mathbf{C}_2\))=1. For \(N\geq3\), in explicit form, we have \((\mathbf{C}_N)_{{N\choose 2}\times N}:=
\left(\begin{smallmatrix} 1 & 1 & 0 & \cdots & \cdots & \cdots & 0\\ 
1 & 0 & 1 & 0 & \cdots & \cdots & 0\\
 1 & 0 & 0 & 1 & 0 & \cdots & 0\\
 \vphantom{\int^0}\smash[t]{\vdots} & & & \vphantom{\int^0}\smash[t]{\vdots} &  &  & \vphantom{\int^0}\smash[t]{\vdots}\\
 1 & 0 & \cdots & 0 & 0 & 1 & 0\\
 1 & 0 & \cdots & \cdots & 0 & 0 & 1\\
 0 & 1 & 1 & 0 & \cdots & \cdots & 0\\
 \hline
 & & & \mathbf{P}_N &  &  & 
 \end{smallmatrix}\right)\) and \(\mathbf{P}_N\) is a placeholder matrix. We convert it to its row echelon form, to obtain the matrix \( \left(\begin{smallmatrix} 1 & 1 & 0 & \cdots & \cdots & \cdots & 0\\ 
0 & 1 & -1 & 0 & \cdots & \cdots & 0\\
 0 & 0 & 1 & -1 & 0 & \cdots & 0\\
 \vphantom{\int^0}\smash[t]{\vdots} & & & \vphantom{\int^0}\smash[t]{\vdots} &  &  & \vphantom{\int^0}\smash[t]{\vdots}\\
 0 & \cdots & \cdots & 0 & 1 & -1 & 0\\
 0 & \cdots & \cdots & \cdots & 0 & 1 & -1\\
 0 & \cdots & \cdots & \cdots & \cdots & 0 & 1\\
 \hline
 & & & \mathbf{0}_N &  &  & 
 \end{smallmatrix} \right)\), where $\mathbf{0}_N$ is a zero matrix. 
 Clearly, there are \(N\) non-zero rows, and hence rank\((\mathbf{C}_N)=N\).
  \end{proof}

\begin{lemma} 
\label{det}
For \(N\geq 3\), determinant of \(\mathbf{C}_N^{[1:N]}\) is \(2\) for even \(N\) and \(-2\) for odd \(N\).  
\end{lemma}

\begin{proof}
We convert \(|\mathbf{C}_N^{[1:N]}|\) to upper triangular form, so that we need to evaluate \(\textrm{det}\left(\begin{smallmatrix} 1 & 1 & 0 & \cdots & \cdots & \cdots & 0\\ 
0 & -1 & 1 & 0 & \cdots & \cdots & 0\\
 0 & 0 & -1 & 1 & 0 & \cdots & 0\\
 \vphantom{\int^0}\smash[t]{\vdots} & & & \vphantom{\int^0}\smash[t]{\vdots} &  &  & \vphantom{\int^0}\smash[t]{\vdots}\\
 0 & 0 & \cdots & 0 & -1 & 1 & 0\\
 0 & 0 & \cdots & \cdots & 0 & -1 & 1\\
 0 & 0 & 0 & 0 & \cdots & \cdots & 2\\
\end{smallmatrix}\right)\). Multiplying the diagonal elements we see \(|\mathbf{C}_N^{[1:N]}|=-2\) for odd \(N\), and \(|\mathbf{C}_N^{[1:N]}|=2\) for even \(N\).
\end{proof}

\begin{sloppypar}
Since $\mathbf{C}_N^{[1:N]}$ is not singular, we invert the matrix to obtain

\begin{equation}
\label{inverse_result}
    (\mathbf{C}_N^{[1:N]})^{-1}=\frac{1}{2}\left(
\begin{smallmatrix} 1 & 1 & 0 & \cdots & \cdots & 0 & -1\\ 
1 & -1 & 0 & \cdots & \cdots & 0 & 1\\
 -1 & 1 & 0 & \cdots & \cdots & 0 & 1\\
 -1 & -1 & 2 & 0 & 0 & \cdots & 1\\
  -1 & -1 & 0 & 2 & 0 & \cdots & 1\\
 \vphantom{\int^0}\smash[t]{\vdots} & & & \vphantom{\int^0}\smash[t]{\vdots} &  &  & \vphantom{\int^0}\smash[t]{\vdots}\\
 -1 & -1 & \cdots & \cdots & 2 & 0 & 1\\
 -1 & -1 & 0 & \cdots & \cdots & 2 & 1\\
 \end{smallmatrix}\right). 
\end{equation} 

We are now ready to give the main result of this work. 
\end{sloppypar}
 
\section{Existence of $\textrm{\textrm{RDS}}(N)$s and computational results}\label{main_section}

The central result of this work is given as follows.

\begin{theorem} \label{the:poidis}
 For the parabola, \(y=x^2\), we obtain:
 \begin{enumerate}
     \item infinitely many $\textrm{\textrm{RDS}}(2)$s for each \(\bm{\psi}_2\) obeying the 'Distinct Coordinates' condition.
     \item a unique $\textrm{\textrm{RDS}}(3)$ for each \(\bm{\psi}_3\) obeying the 'Distinct Coordinates' condition.
     \item a unique $\textrm{\textrm{RDS}}(N)$ for each \(\bm{\psi}_N\) that obeys the 'Distinct Coordinates' and 'Existence Condition'; otherwise, no such set exists.
 \end{enumerate}
 
For brevity, let the entries of \(\bm{\psi}_N\) be $\psi_i\;\; (1\leq i\leq {N\choose 2})$. Then, the $x$-coordinates of the $\textrm{\textrm{RDS}}(N)$ is given by:

\begin{equation}
\label{exact_soln}
    x_i=\frac{1}{2}
    \begin{cases}
    (\psi_1+\psi_2-\psi_N) \;\; \textrm{if}\; i=1\\ 
    (\psi_1-\psi_2+\psi_N) \;\; \textrm{if}\; i=2\\
    (-\psi_1+\psi_2+\psi_N) \;\; \textrm{if}\; i=3\\
    (-\psi_1-\psi_2+2\psi_{i-1}+\psi_N) \;\; \textrm{if }\; 4\leq i\leq N.
    \end{cases} 
\end{equation}.

The 'Distinct Coordinates' condition is given by,

\begin{equation}
\label{dist_cond}
    \begin{cases}
    \psi_{1}\neq\psi_{2},\psi_{1}\neq\psi_{N} \; \textrm{or} \;\psi_{2}\neq\psi_{N} \;\; \textrm{if}\; 1\leq i<j\leq 3\\
    \psi_{i-1}\neq\psi_{j-1} \;\; \textrm{if}\; 4\leq i<j\leq N\\ 
    \psi_{1}\neq\psi_{j-1}, \psi_{2}\neq\psi_{j-1} \; \textrm{or} \; \psi_{1}+\psi_2\neq\psi_{j-1}+\psi_N\;\; \textrm{if}\; 1\leq i\leq 3, 4\leq j\leq N.\\
    
    \end{cases}
\end{equation}.

and the 'Existence Condition' is given by,

\begin{equation}
\label{exist_condition}
    \begin{cases}
    \psi_{2}+\psi_{N+i}=\psi_N+\psi_{i+2} \;\; \textrm{if}\; 1\leq i\leq N-3\\
   \psi_{1}+\psi_{N+i}=\psi_N+\psi_{i+5-N} \;\; \textrm{if}\; N-2\leq i\leq 2N-6\\
   \psi_{1}+\psi_2+\psi_{N+i}=\psi_N+\psi_{m_{(i+6-2N)}+2}+\psi_{n_{(i+6-2N)}+2} \;\; \textrm{if}\; 2N-5\leq i\leq {N\choose 2}-N\\
    
    \end{cases}
\end{equation}

where \((m_i,n_i)\) is the $i^{th}$ 2-tuple of the \(N_{(2)}\)-Indices set.  
\end{theorem}
\begin{sloppypar}
 \begin{proof}
Assuming that an $\textrm{\textrm{RDS}}(N)$ exists, notice that we must have a solution to Equation \ref{mat_eq1}. Since $\mathbf{x}_N=(\mathbf{C}_N^{[1:N]})^{-1}\bm{\psi}_N^{[1:N]}$, using Equation \ref{inverse_result}, we obtain the exact form of the solutions in Equation \ref{exact_soln}. Each of the $x$-coordinates however must be distinct and so we can apply $x_i\neq x_j$ for $1\leq i<j\leq N$ and obtain the 'Distinct Coordinates' condition in Equation \ref{dist_cond}. 

Now we investigate the existence of these solutions. Look at the augmented matrix \([\mathbf{C}_N|\bm{\psi}_N]\). For \(N=2\), we find rank(\([\mathbf{C}_N|\bm{\psi}_N])=\)  rank(\(\mathbf{C}_N)=1<N (=2)\) for any \(\bm{\psi}_2\). Hence, there are infinitely many solutions \(\mathbf{x}_{2}\) to this system for each \(\bm{\psi}_2\) obeying 'Distinct Coordinates'.  

For \(N\geq3\), we claim that rank(\([\mathbf{C}_N|\bm{\psi}_N])\geq\)  rank(\(\mathbf{C}_N)\). Since by Lemma \ref{lem:rankofC}, rank(\(\mathbf{C}_N)=N\), we have (\({N\choose 2}-N\)) zero rows in the row echelon form of \(\mathbf{C}_N\). In the augmented matrix, however, the entries of these (\({N\choose 2}-N\)) rows are linear combinations of the entries of  \(\bm{\psi}_N\) which need not necessarily equal zero. Hence, our claim is true.
For \(N=3\), rank equality occurs and hence a unique solution \(\mathbf{x}_{3}\) is obtained for each \(\bm{\psi}_3\) obeying 'Distinct Coordinates'. However, for \(N>3\), rank(\([\mathbf{C}_N|\bm{\psi}_N])>\)  rank(\(\mathbf{C}_N)\) implies that, in general, the system has no solution. We can circumvent this if we can choose the entries of \(\bm{\psi}_N\), such that the non-zero entries of these rows are set to zero by design. If successful to find such a set of Pythagorean ratios, we obtain rank(\([\mathbf{C}_N|\bm{\psi}_N])=\)  rank(\(\mathbf{C}_N)\) and thus a unique solution \(\mathbf{x}_{N}\). 

Explicitly, this procedure amounts to satisfying the 'Existence Condition'. We consider the equation \(\mathbf{M}_N\mathbf{C}_N=\mathbf{C}^{\textrm{rref}}_N\), where $\mathbf{C}^{\textrm{rref}}_N$ is the coefficient matrix in row reduced echelon form (RREF). Then, $\mathbf{M}_N$ is the product of the elementary row operation matrices to row reduce matrix $\mathbf{C}_N$ and is given by $
\left(
\begin{array}{c|c}
(\mathbf{C}_N^{[1:N]})^{-1} & \mathbf{0}_{{N\choose 2} -N}  \\
\hline
\begin{array}{ccc}
  \mathbf{(01)}_{N-3}   & -\mathbf{I}_{N-3}  & -\mathbf{(1)}_{N-3} \\
     \mathbf{(10)}_{N-3} & -\mathbf{I}_{N-3} & -\mathbf{(1)}_{N-3}\\
     \mathbf{(11)}_{{N-3\choose 2}} & -\mathbf{C}_{N-3} &-\mathbf{(1)}_{{N-3\choose 2}}
\end{array}
 & \mathbf{I}_{{N\choose 2} -N}
\end{array}
\right)$, where \(\mathbf{(01)}_{N-3}\) is the \((N-3)\times 2\) vector with all the entries of the first and second column are 0 and 1 respectively. The $\mathbf{(10)}$ and $\mathbf{(11)}$ blocks below it are defined similarly. \(\mathbf{(1)}_{N-3}\) (resp.  \(\mathbf{(1)}_{{N-3}\choose 2}\)) is the \((N-3)\times 1\) (resp. ${{N-3}\choose 2}\times 1$) vector with all entries equal to 1. Now, observe that the existence condition is equivalent to \(\mathbf{M}_N^{\left[N+1:{N\choose 2}\right]}\bm{\psi}_N=\mathbf{0}\).  This gives us the 'Existence Condition' in Equation \ref{exist_condition}. 

\end{proof}
\end{sloppypar}

The above theorem is very encouraging. It maps the problem of finding $\textrm{RDS}(N)$s on a parabola, to a problem of finding sets of Pythagorean triplets obeying certain properties. This is a useful connection and can be applied to prove results pertaining to such rational distance sets, by using properties of the Pythagorean triplets.

An additional constraint often discussed for $\textrm{RDS}(N)$s is that the set should be in 'General Position', that is, they should be such that no three lie on a line and no four lie on a circle. The first constraint is met automatically on the parabola, and thus, this constraint is equivalent to a condition of non-concyclicity on the parabola. We thus obtain the additional 'General Position' condition given by
\begin{equation}
    \begin{cases}
    \psi_{s-1}+\psi_N\neq 0 \;\; \textrm{if}\; p=1, q=2, r=3, 4\leq s\leq N\\
    \psi_{p-1}+\psi_{q-1}+\psi_{r-1}+\psi_{s-1}\neq2(\psi_{1}+\psi_{2}-\psi_N) \;\; \textrm{if}\; 4\leq p<q<r<s\leq N\\ 
    
    \psi_{q-1}+\psi_{r-1}+\psi_{s-1}\neq\psi_{1}+\psi_2-\psi_N, \psi_{q-1}+\psi_{r-1}+\psi_{s-1}\neq\psi_{1}+2\psi_2-2\psi_N \; \textrm{or} \;\\\hspace{75pt} \psi_{q-1}+\psi_{r-1}+\psi_{s-1}\neq2\psi_{1}+\psi_2-2\psi_N\;\; \textrm{if}\; 1\leq p\leq 3, 4\leq q<r<s\leq N\\
    \psi_{r-1}+\psi_{s-1}\neq\psi_{1}-\psi_N, \psi_{r-1}+\psi_{s-1}\neq\psi_{2}-\psi_N  \; \textrm{or} \; \psi_{r-1}+\psi_{s-1}\neq\psi_1+\psi_{2}-2\psi_N\;\;\\\hspace{200pt} \textrm{if}\; 1\leq p<q\leq 3, 4\leq r<s\leq N\\ 
    \end{cases}
    \label{general_position}
\end{equation}

The explicit form of the solutions, and the accompanying 'Distinct Coordinates' and 'General Position' condition for small $N$ ($\leq 5$) has been tabulated in Table \ref{tab:explicit_soln}.

\begin{table}
\begin{center}
\caption{Explicit solutions for \(2\leq N \leq 5\), and the conditions for 'Distinct Coordinates' and 'General Position' can easily be obtained from Equations ~\eqref{dist_cond} and \eqref{general_position}.}
\label{tab:explicit_soln}
\begin{tabular}{c c } 
 \hline
$N$ & $\mathbf{x}_N$ \\
 \hline
 2 & $\begin{pmatrix}
           r \\
           \psi_{12}-r
         \end{pmatrix}$, $r\in \mathbb{Q}$ \\ 
 3 & $\frac{1}{2}\begin{pmatrix}
           \psi_{12}+\psi_{13}-\psi_{23} \\
           \psi_{12}-\psi_{13}+\psi_{23}\\
           -\psi_{12}+\psi_{13}+\psi_{23}
         \end{pmatrix}$\\
 4 & $\frac{1}{2} \begin{pmatrix}
           \psi_{12}+\psi_{13}-\psi_{14} \\
           \psi_{12}-\psi_{13}+\psi_{14}\\
           -\psi_{12}+\psi_{13}+\psi_{14}\\
           -\psi_{12}-\psi_{13}+2\psi_{23}+\psi_{14}
         \end{pmatrix}$
 \\
 5 &$\frac{1}{2} \begin{pmatrix}
           \psi_{12}+\psi_{13}-\psi_{14} \\
           \psi_{12}-\psi_{13}+\psi_{14}\\
           -\psi_{12}+\psi_{13}+\psi_{14}\\
           -\psi_{12}-\psi_{13}+2\psi_{13}+\psi_{14}\\
             -\psi_{12}-\psi_{13}+2\psi_{13}+\psi_{14}
         \end{pmatrix}$ \\
 \hline
\end{tabular}
\end{center}
\end{table}

Theorem \ref{the:poidis} also provides us a prescription to computationally determine examples of $\textrm{RDS}(N)$s. Given a positive real $\Gamma$ a priori, we can generate Pythagorean triplets with hypotenuse bound by $\Gamma$, and store the corresponding Pythagorean ratios. Let the number of such ratios be $\mathcal{A}(\Gamma)$. Iterating through all possible $N\choose 2$ combinations of the Pythagorean ratios and applying the 'Existence' and 'Distinct Condition' to determine which of them produce valid $\textrm{RDS}(N)$s is clearly computationally expensive. For instance, we know that $\mathcal{A}(\Gamma)\sim\frac{\Gamma}{2\pi}$ and thus, the number of combinations are of the order of ${\frac{\Gamma}{2\pi}+{{\frac{\Gamma}{2\pi}}\choose 2}-1}\choose {{\frac{\Gamma}{2\pi}}\choose 2}$ (since certain repetitions are allowed), that is, increases exponentially in $\Gamma$. Thus, it is paramount that we decrease the number of iterations for searching, and this may be achieved by the following observation. 

\begin{lemma} \label{lem:ind}
 Given that $\bm{\psi}_N$ obeys the 'Existence Condition', $\psi_{N+i}$ is a linear combination of  $\psi_{1}$ to $\psi_{N}$ for $1\leq i\leq {N\choose 2}-N$, and is given explicitly by:
 
 \begin{equation}
 \label{ind_dep}
    \psi_{N+i}=\begin{cases}
    \psi_N+\psi_{i+2}-\psi_{2} \;\; \textrm{if}\; 1\leq i\leq N-3\\
   \psi_N+\psi_{i+5-N}-\psi_{1} \;\; \textrm{if}\; N-2\leq i\leq 2N-6\\
   \psi_N+\psi_{m_{(i+6-2N)}+2}+\psi_{n_{(i+6-2N)}+2}-\psi_{1}-\psi_2 \;\; \textrm{if}\; 2N-5\leq i\leq {N\choose 2}-N\\
    
    \end{cases}
\end{equation}
\end{lemma}
\begin{proof}

Observe that Equation \ref{exist_condition} may be rewritten in the form of Equation \ref{ind_dep}. We now need to show the indices on the right hand side are all bounded by $N$. 

We see that $(i+2)\leq (N-1)$ for $i\leq N-3$, $(i+5-N)\leq (N-1)$ for $i\leq 2N-6$, therefore, $\psi_{N+i}$ can be written completely in terms of $\psi_{1}$ to $\psi_{N}$ for $1\leq i\leq 2N-6$. The third case occurs for $N\geq 5$. Now, since $m_i\leq n_i$ in general, we need only consider $n_{i+6-2N}+2\leq n_{{N\choose 2}+6-3N}+2=n_{{{N-3}\choose 2}}+2$ for $N\geq 5$. Now, $\textrm{max}_{1\leq i\leq {N\choose 2}}n_i=N$, and hence $n_{{{N-3}\choose 2}}+2\leq N+2$. Now, $N+2\leq N+(2N-6)$ for $N\geq 4$, and so, we can express $\psi_{N+i}$  for $2N-5\leq i \leq {N\choose 2}-N$ also in terms of $\psi_1$ to $\psi_N$.      
  \end{proof}

The above result provides a fascinating reinterpretation of the problem. A useful Pythagorean ratio vector thus must consist of an independent and dependent part, where \(\bm{\psi}_{N}^{[1:N]}\) is the independent Pythagorean ratio vector, while  \(\bm{\psi}_{N}^{[N+1:{N\choose 2}]}\) is the dependent Pythagorean ratio vector.

\begin{sloppypar}Thus, our algorithm involves constructing an $N$ combination of the Pythagorean ratios and then using Equation \ref{ind_dep} to construct candidate dependant Pythagorean ratios, and then check if these values are indeed valid Pythagorean ratios. If they are, we compute the $x$ coordinates of the corresponding $\textrm{RDS}(N)$ using Equation 5. The 'Distinct Coordinates' condition can be applied computationally either in terms of the $x$ coordinates of the candidate $\textrm{RDS}(N)$, or in terms of the Pythagorean ratios. Iterating through all possible $N$ combinations of the stored Pythagorean ratios for each $\Gamma$, we can determine the corresponding $\textrm{RDS}(N)$s (if any).
\end{sloppypar}

A Python code implementing this algorithm was implemented parallely on a GPU (with upto $\sim25$ GB RAM, accessed via Google Colab). In Table \ref{tab:examples_soln}, we present 10 new examples for $N=3, 4$ and $5$, as an illustration and the corresponding $\bm{\psi}_N$s that are used to construct the $\textrm{RDS}(N)$s using our algorithm. We also enumerate the number of $\textrm{RDS}(N)$s found for $\Gamma\leq 150$ and enumerate them explicitly in Table \ref{tab:number} for $N=3$ and $4$. 



\begin{table}
\begin{center}
\caption{Example (new) solutions $x_N$ for \(3\leq N \leq 5\), along with corresponding Pythagorean ratio vector $\Psi_N$.}
\label{tab:examples_soln}
\begin{tabular}{ccc}
\hline
$N$ & Pythagorean ratio vector ($\Psi_{N}$) & Solution ($\mathbf{x}_N$) \\
\hline
\multirow{5}{*}{3}&& \\
&$\begin{pmatrix}\frac{4}{3}&-\frac{5}{12}&\frac{5}{12}\end{pmatrix}^T$&$\begin{pmatrix}\frac{1}{4} &\frac{13}{12} &-\frac{2}{3}\end{pmatrix}^T$\\
&& \\
&$\begin{pmatrix}-\frac{4}{3} &\frac{12}{5}&-\frac{8}{15}\end{pmatrix}^T$&$\begin{pmatrix}\frac{38}{15} &-\frac{6}{15}&-\frac{2}{15}\end{pmatrix}^T$ \\
&& \\
\hline
\multirow{9}{*}{4}&& \\
&$\begin{pmatrix}-\frac{35}{12} &-\frac{4}{3}&-\frac{7}{24}&-\frac{3}{4}&\frac{7}{24}&\frac{15}{8}\end{pmatrix}^T$&$\begin{pmatrix}-\frac{7}{4},-\frac{7}{6},\frac{5}{12},\frac{35}{24}\end{pmatrix}^T $\\

&& \\
&$\begin{pmatrix}-\frac{208}{105} &-\frac{20}{21}&\frac{208}{105}&-\frac{8}{15}&\frac{15}{5}&\frac{24}{7}\end{pmatrix}^T$&$\begin{pmatrix}-\frac{6}{5}&-\frac{82}{105}&\frac{26}{105}&\frac{334}{105}\end{pmatrix}^T$ \\

&& \\
&$\begin{pmatrix}\frac{95}{168} &\frac{20}{21}&\frac{45}{28}&\frac{65}{168}&\frac{25}{24}&\frac{35}{12}\end{pmatrix}^T$&$\begin{pmatrix}-\frac{5}{28}&\frac{125}{168}&\frac{95}{84}&\frac{25}{14}\end{pmatrix}^T$ \\

&& \\
&$\begin{pmatrix}-\frac{15}{8}&-\frac{21}{20}&-\frac{8}{15}&-\frac{7}{24}&\frac{9}{40}&\frac{21}{20}\end{pmatrix}^T$&$\begin{pmatrix}-\frac{79}{60}&-\frac{67}{120}&\frac{4}{15}&\frac{47}{60}\end{pmatrix}^T$ \\

&& \\
&$\begin{pmatrix}-\frac{779}{660}&-\frac{371}{264}&\frac{7}{24}&\frac{9}{40}&\frac{21}{20}&\frac{56}{33}\end{pmatrix}^T$&$\begin{pmatrix}-\frac{853}{880}&-\frac{557}{2640}&\frac{1151}{2640}&\frac{3329}{2640}\end{pmatrix}^T $\\

&& \\
&$\begin{pmatrix}-\frac{35}{12}&-\frac{21}{20}&-\frac{9}{40}&-\frac{7}{24}&\frac{8}{15}&\frac{12}{5}\end{pmatrix}^T$&$\begin{pmatrix}-\frac{147}{80}&-\frac{259}{240}& \frac{63}{80}, \frac{129}{80}\end{pmatrix}^T$\\

&& \\
&$\begin{pmatrix}-\frac{28}{45}&\frac{11}{60}&\frac{48}{55}&\frac{85}{132}&\frac{4}{3}&\frac{77}{36}\end{pmatrix}^T$& $\begin{pmatrix}-\frac{268}{495}&-\frac{8}{99}&\frac{287}{396}&\frac{140}{99}\end{pmatrix}^T$ \\

&& \\
\hline
\multirow{3}{*}{5}&& \\
&$\begin{pmatrix}0&\frac{7}{24}&\frac{4}{3}&-\frac{3}{4}&-\frac{7}{24}&\frac{3}{4}&-\frac{4}{3}&\frac{25}{24}&-\frac{25}{24}&0\end{pmatrix}^T$&$\begin{pmatrix}\frac{7}{24}&-\frac{7}{24}&0&\frac{25}{24}&-\frac{25}{24}\end{pmatrix}^T$\\
&& \\
\hline

\end{tabular}
\end{center}
\end{table}

\begin{table}
\begin{center}
\caption{Number of solutions $x_N$ generated for all possible valid $\Psi_N$s such that $\gamma \leq \Gamma$ for all Pythagorean ratios in the independent vector $\Psi_N^{[1:N]}$. The total number of solutions $\theta_N^{\textrm{all}}$ and the number of solutions in general position $\theta_N^{\textrm{gp}}$ are enlisted for $3 \leq N\leq 4$. The limits $25\leq \Gamma\leq 145$ are chosen such that we enlist $\Gamma$ whenever there is an increase in the number of solutions.}
\label{tab:number}
\begin{tabular}{ccccc}
\hline
 \multirow{2}{*}{Limit}& \multicolumn{2}{c}{$N=3$} & \multicolumn{2}{c}{$N=4$} \\ 
\cline{2-5} \
 \multirow{2}{*}{($\Gamma)$}&\multirow{2}{*}{$\theta_3^{\textrm{gp}}(\Gamma)$}&\multirow{2}{*}{$\theta_3^{\textrm{all}}(\Gamma)$}&\multirow{2}{*}{$\theta_4^{\textrm{gp}}(\Gamma)$}&\multirow{2}{*}{$\theta_4^{\textrm{all}}(\Gamma)$}\\ 
 & & & & \\
 \hline
25&672&680&16&176\\
29&1320&1330&36&334\\
41&3640&3654&40&883\\
53&5440&5456&88&1328\\
61&7752&7770&108&1893\\
65&14168&14190&148&3459\\
73&18400&18424&180&4504\\
85&29232&29260&228&7159\\
89&35960&35990&256&8826\\
97&43648&43680&288&10704\\
101&52360&52394&316&12855\\
109&62156&62196&392&15302\\
113&73108&73150&420&17999\\
125&85276&85320&432&20972\\
137&98724&98770&500&24321\\
145&129716&129766&544&31941\\
\hline
\end{tabular}
\end{center}
\end{table}

\section{Density of the $\textrm{RDS}(N)$}\label{density_section}

The power of our analysis in Section \ref{prelimiaries} can be realized by its application to density analysis of the $\textrm{RDS}(N)$s on the parabola. Our approach entails taking advantage of the density of the Pythagorean ratios in \(\mathbb{R}\) and then using the linear algebraic correspondence to show the density of $\textrm{RDS}(N)$ for $N=2$ and $3$. For this purpose, we first recall a triplet counting function given by Hinson \cite{E_Hinson}.

\begin{definition}\label{count_func}[Counting Function]
The counting function \(\nu\) for positive naturally-ordered primitive Pythagorean triplets (\(\alpha,\beta,\gamma\)) is such that 
\(\nu(\frac{p}{q}):S\to\{0,1\}\) where \(S=\mathbb{Q}\cap(0,1)\), and \(\nu(\frac{p}{q})=1\) only when there exists a solution \((p,q)\) in \(\mathbb{N}\times\mathbb{N}\) to \(\alpha^2+\beta^2=\gamma^2\), where \(\beta=\alpha+p\) and \(\gamma=\alpha+q\).
\end{definition}

Hinson \cite{E_Hinson} shows that there exists a one-to-one correspondence between the elements of \(\nu^{-1}(1)\) (that is, the rationals \(\frac{p}{q}\)) and the positive naturally-ordered primitive
Pythagorean triplets (\(\alpha,\beta,\gamma\)). Further, he provides the following result:

\begin{lemma}[Hinson] \label{lem:dense triples}
\(\nu^{-1}(1)\) is dense in the real unit interval
\([0,1]\).
\end{lemma}

\begin{sloppypar}
We now need the following definitions about the sets of Pythagorean ratios.
\end{sloppypar}

\begin{definition}[Pythagorean Ratio Sets]\label{def:ratio_set}
We define a Pythagorean ratio set \(\Psi\), consisting of all possible Pythagorean ratios. The corresponding elements of these sets will be denoted in small cases. We provide a notation for the subsets of \(\Psi\) as follows. 
\begin{enumerate}
    \item  \(\Psi^+_{order}\): for positive naturally-ordered primitive Pythagorean triplets.
    \item  \(\Psi^+_{opp-order}\): for positive oppositely-ordered primitive Pythagorean triplets.
    \item  \(\Psi^+\): for positive primitive Pythagorean triplets.
    \item \(\Psi^-\): for negative primitive Pythagorean triplets.
\end{enumerate}
\end{definition}

By a successive extension of the density result in Lemma \ref{lem:dense triples} (ordered to positive to general primitive Pythagorean triplets) using elementary results from real analysis, we can show that the set of Pythagorean ratios is dense in the real line. This is detailed in the following theorem.

\begin{theorem}
\label{psi_dense}
The set \(\Psi\) is dense in the real interval \(\mathbb{R}\).
\end{theorem}
\begin{proof}

We first show that \(\Psi^+_{order}\) is dense in the real interval
\([1,\infty)\). Rewrite \(\psi^+_{order}=\frac{\beta}{\alpha}\) (from Definition \ref{count_func}) in terms of \(p\) and \(q\). Solving for \(\alpha\) in \(\alpha^2+(\alpha+p)^2=(\alpha+q)^2\) and noting that \(q>p\), we obtain \(\alpha=q-p+(2q(q-p))^\frac{1}{2}\).  We thus have \(\psi^+_{order}=\frac{1+(2(1-\frac{p}{q}))^\frac{1}{2}}{1-\frac{p}{q}+(2(1-\frac{p}{q}))^\frac{1}{2}}\). Now, the rationals \(\frac{p}{q}\) are elements of the set \(\nu^{-1}(1)\) which by Lemma \ref{lem:dense triples} is dense in \([0,1]\). Notice that the function \(f_1: [0,1]\to[1,\infty)\) given by $f_1(x)=\frac{1+(2(1-x))^\frac{1}{2}}{1-x+(2(1-x))^\frac{1}{2}}$ is continuous and surjective. Hence, we have \(\Psi^+_{order}\) dense in the interval \([1,\infty)\). 

Next, we show that \(\Psi^+\) is dense in \((0,\infty)\). Observe that each element of the set \(\Psi^+_{opp-order}\) is the inverse of the element of the set \(\Psi^+_{order}\). Thus, we have \(\psi^+_{opp-order}=\frac{1}{\psi^+_{order}}\). We see again that \(f_2:[1,\infty)\to(0,1]\) is continuous and surjective, given by $f_2(x)=\frac{1}{x}$ and since \(\Psi_{order}^+\) is dense in \([1,\infty)\), \(\Psi^+_{opp-order}\) is dense in \((0,1]\). Thus \(\Psi^+ = \Psi^+_{order} \; \cup \; \Psi^+_{opp-order}\)  is dense in \((0,1] \cup [1,\infty)\), that is, \((0,\infty)\).  

Finally, we claim that \(\Psi\) is dense in \(\mathbb{R}\). Observe that each element of \(\Psi^-\) is the negative of the elements of the set \(\Psi^+\). We thus have \(\psi^-=-\psi^+\). Thus \(f_3:(0,\infty)\to(-\infty,0)\) given by $f_3(x)=-x$ is continuous and surjective, and since \(\Psi^+\) is dense in \((0,\infty)\), \(\Psi^-\) is dense in \((-\infty,0)\). Thus we have that \(\Psi^+ \cup \Psi^-\) is dense in \(\mathbbm{R}-\{0\}\). Now, we can add the singleton set \(\{\psi_0\}=\{0\}\) to the dense set \(\Psi^+ \cup \Psi^-\) to show that the set \(\Psi\) is dense in \(\mathbb{R}\).  
\end{proof}

Next, we wish to show the density of the $N$-tuples of Pythagorean ratios that we use to construct the $\textrm{RDS}(N)$s. To do so, we need the following result from real analysis. 

 \begin{lemma}
 \label{hyperplane}
Given $L$ to be the set of points on a $N-1$ dimensional hyperplane in $\mathbb{R}^N$, and a set $\mathcal{A}$ dense in $\mathbb{R}^N$, the set  $\mathcal{A}\cap (\mathbb{R}^N\backslash L)$ is dense in $\mathbb{R}^N$.
\end{lemma}

 \begin{proof}
We first observe that the plane $L$ partitions $\mathbb{R}^N$ to three sets $\mathcal{D}_1$ and $\mathcal{D}_2$ (half spaces) and $\mathcal{D}_3:= L$. Now, clearly $\mathcal{A}\cap (\mathbb{R}^N\backslash L)$ is dense in $\mathbb{R}^N\backslash L$, since $\mathcal{A}$ is dense in $\mathbb{R}^N$. We claim that this also implies that $\mathcal{A}\cap (\mathbb{R}^N\backslash L)$ is dense in $\mathbb{R}^N$. This is because, if we consider any open ball $B\in \mathbb{R}^N$, then there must exist an open ball $B'\in B$ such that $B'\in (\mathbb{R}^N\backslash L)$. Now, since  $\mathcal{A}\cap (\mathbb{R}^N\backslash L)$ is dense in $\mathbb{R}^N\backslash L$, therefore the set $B'\cap \mathcal{A}$ must be non-empty. This also implies that $B\cap \mathcal{A}$ is non-empty for any open ball in $\mathbb{R}^N$. Hence, our claim is true. 
 \end{proof}

We define two restrictions to the set $\bm{\Psi}^N$. Call the set $\Psi^N_{DC}$ to be the set of $N$-tuples of Pythagorean ratios that obey the \textit{Distinct Coordinates} condition for an $\textrm{RDS}(N)$. Also, we call the set $\Psi^N_{DC, EC}$ to be the set of $N$-tuples of Pythagorean ratios that obey the \textit{Distinct Coordinates} and \textit{Existence Condition} for $\textrm{RDS}(N)$.  We show the following result.  
 
 \begin{lemma}
For $N\geq 3$, $\Psi^N_{DC}$ is dense in $\mathbb{R}^N$.
 \end{lemma}
 
 \begin{proof}
 From Theorem \ref{psi_dense},  $\Psi^N$ must be dense in $\mathbb{R}^N$. Let an element of  $\Psi^N$ be given by $\left(\psi_1 \; \psi_2 \; \hdots \; \psi_N\right)^T$. Now, consider the set $\mathcal{L}_{DC}\subset \mathbb{R}^N$ given by $\mathcal{L}_{DC}:=\{\mathbf{u}=\left(u_1 \; u_2 \; \hdots \; u_N\right)^T\in  \mathbb{R}^N |  u_{1}= u_{2},u_{1}= u_{N} \; \textrm{and} \;u_{2}= u_{N} \; \textrm{if}\; 1\leq i<j\leq 3; u_{i-1}= u_{j-1} \; \textrm{if}\; 4\leq i<j\leq N;  u_{1}= u_{j-1}, u_{2}= u_{j-1} \; \textrm{and} \; u_{1}+u_2= u_{j-1}+u_N\;\; \textrm{if}\; 1\leq i\leq 3, 4\leq j\leq N\}.$  Notice that $\Psi^N_{DC}=\Psi^N\cap (\mathbb{R}^N\backslash \mathcal{L}_{DC})$ (recall Equation \ref{dist_cond}). Observe that $\mathcal{L}_{DC}$ is the union of $N\choose 2$ hyperplanes in $\mathbb{R}^N$. Thus, by repeated application of Lemma \ref{hyperplane}, we conclude $\Psi^N_{DC}$ is dense in $\mathbb{R}^N$.
\end{proof} 

It has not been possible to show a density result for $\Psi_{DC,EC}$ yet, for general $N$. We thus pose the following open question.

\begin{question}
\label{density_of_pyth}
For $N\geq 4$, if there are infinitely many $\textrm{RDS}(N)$s, is $\Psi_{DC,EC}^N$ dense in $\mathbb{R}^N$? In particular, is $\Psi_{DC,EC}^4$ dense in $\mathbb{R}^4$?
\end{question}

Discussion on the above question is done in Section 5. Nevertheless, it is possible to make progress for $N=2$ and 3, since in these cases, $\Psi_{DC,EC}^N=\Psi_{DC}^N$. Defining a function $C_N^{[1:N]}$ corresponding to the restricted coefficient matrix $\mathbf{C}_N^{[1:N]}$ (recall Equation \ref{mat_eq1}), we can provide the following result.

\begin{lemma} 
For $N=2$, $C_2:\mathbb{R}^2\rightarrow \mathbb{R}$ is an open map. For $N\geq 3$, the map \(C_N^{[1:N]}:\mathbb{R}^N\rightarrow\mathbb{R}^N\) is an open map. 
\end{lemma}

\begin{proof}
We do the $N=2$ case using first principles. Consider a point \(X=(u_1,u_2)\) in \(\mathbb{R}^2\). Construct an open ball \(U\) centered at \(X\) of radius \(\epsilon\). Any arbitrary point \(X'\) lying in \(U\) can be written as \((u_1+\epsilon_1,u_2+\epsilon_2)\), where \(\sqrt{\epsilon_1^2+\epsilon_2^2}<\epsilon\). Now, call \(C_2(X)\) as \(Y\) in \(\mathbb{R}\). Thus, \(Y=(u_1+u_2)=y\) (say). Now call \(C_2(X')\) as \(Y'\) and see that \(Y'=(u_1+u_2+\epsilon_1+\epsilon_2)\). Thus  \(Y'\) lies inside an open interval \(V\) of radius \(\epsilon\sqrt{2}\) in \(\mathbb{R}\), for any \(X'\) in \(U\). Thus, we have shown that all points in the open ball \(U\) can be mapped to the inside of an appropriately sized open interval \(V\) in \(\mathbb{R}\). 

We now show that this \(V\) is indeed \(C_2(U)\). Choose an arbitrary point \(Z'=(y+\delta)\) at a distance \(\delta\) from \(Y\) inside \(V\). Thus \(\delta<\epsilon\sqrt{2}\). We show that \(Z'\) is the image of a point \(Z=(u_1+\omega_1,u_2+\omega_2)\) that lies in \(U\). For this to be true, we must have \(\omega_1+\omega_2=\delta\), which implies that we have to choose \(\omega_1\) and \(\omega_2\) such that \(\sqrt{\omega_1^2+\omega_2^2}<\epsilon\). This can be obtained by choosing \(\omega_1=\omega_2=\frac{\delta}{2}\). Thus for \(N=2\), \(V\) is indeed \(C_2(U)\). 

Now, since any arbitrary open set is a union of open balls, we have shown that \(C_2\) maps open sets in \(\mathbb{R}^2\) to open sets in \(\mathbb{R}\). Thus, \(C_2\) is an open map.

For $N\geq 3$, the map \(C_N^{[1:N]}\) maps \((x_1,x_2,\hdots,x_{N-1},x_N)\rightarrow (x_1+x_2, x_1+x_3, \hdots, x_1+x_N, x_2+x_3)\). The Jacobian of this map is \(\mathbf{C}_N^{[1:N]}\). Using Lemma \ref{det} and by an application of the inverse function theorem, we obtain that \(C_N^{[1:N]}\) is an open map. 
\end{proof}

We thus give the density results for $N=2$ and 3. Let the call \(\mathcal{R}_N\) as the set of all \(N\)-tuples each of which are the \(x\)-coordinates of an $\textrm{RDS}(N)$. Then, showing \(\mathcal{R}_N\) dense in \(\mathbb{R}^N\) is equivalent to showing that the $\textrm{RDS}(N)$ is dense in the set of points of the parabola.

\begin{theorem}
  For $N=2$ and $3$, \(\mathcal{R}_N\) is dense in \(\mathbb{R}^N\). 
\end{theorem}
\begin{proof}

For \(N=2\), we have $\mathbf{C}_2\mathbf{x}_2=\bm{\psi}_2$. Now, since $\bm{\psi}_2\in \Psi$ and \(\Psi\) is dense in \(\mathbb{R}\), we obtain \(\mathcal{R}_2\) to be dense in \(\mathbb{R}^2\).

For $N=3$, consider the function $C_3:\mathbb{R}^3 \rightarrow \mathbb{R}^3$ corresponding to the matrix $\mathbf{C}_3$, and observe that $C_3(\mathbf{x})=C_3^{[1:3]}(\mathbf{x})$ is an open map. Since $\Psi_{DC}^N$ is dense in $\mathbb{R}^3$, by Lemma 1.5 we obtain $\mathcal{R}_3$ to be dense in $\mathbb{R}^3$.
\end{proof}

It is easy to show that the density of $\mathcal{R}_N$ in $\mathbb{R}^N$ depends on whether $\Psi_{DC,EC}^N$ is dense in $\mathbb{R}^N$, since $C_N^{[1:N]}$ is an open map. Thus, the answer to Question \ref{density_of_pyth} will also answer the following open problem. 

\begin{question}
For $N\geq 4$, if there are infinitely many $\textrm{RDS}(N)s$, is $\mathcal{R}_N$ dense in $\mathbb{R}^N$? In particular, is $\mathcal{R}_4$ dense in $\mathbb{R}^4$?
\end{question}

\section{Outlook} \label{outlook}


In this article, we study rational distance sets on the parabola \(y=x^2\). Using a correspondence between the solutions sets of the $\textrm{RDS}(N)$ and the set of primitive Pythagorean triples, we are able to provide a result that gives these solutions in terms of so-called Pythagorean ratios. The existence of an $\textrm{RDS}(N)$ is contingent on the Pythagorean ratios obeying certain properties, which we call the 'Distinct Coordinates' condition. Using this, we are able to give an efficient algorithm that helps to search for new examples of such rational distance sets, and we enlist 10 new examples, including new examples for a 5-point $RDS$, of which only one was known in the previous literature. Furthermore, we demonstrate another use of our formulation in analysis, by showing the density of the  2 and 3-point $\textrm{RDS}$s using a density result known for Pythagorean triplets. 

The extensions of this work can be done in a  threefold manner. One, we can investigate the conditions on the Pythagorean triplets ('Distinct Coordinates' condition), in terms of solving Diophantine equations and see if this structure helps to reveal properties of such rational distance sets. Two, we can perform higher numerics and begin the search for a $6$ point RDS, which will give the first example of such a set. Three, we can attempt to answer Questions 4.1 and 4.2, on the density of rational distance sets with four points. We suspect this will be closely related to the properties of the Pythagorean triplets, and hence, we foresee direction one to be the most promising for future research in this problem. 

\section*{Acknowledgement}
We gratefully acknowledge Ashwin Girish, Ayush Basu and Rachit Bodhare for insightful discussions, and Arkavo Hait and Nallapati Sathvik for helping with writing the parallelized code. We also thank Prof. Santosha Pattanayak, Prof. Santosh Nadimpalli and Prof. Arijit Ganguly for useful feedback on the paper. Initial results were presented in the undergraduate poster session at the AMS Joint Mathematical Meeting 2021, and SB acknowledges the valuable insights provided by the referees. Part of this work was conducted as an IITK Stamatics summer project mentored by SB, who is grateful to the Stamatics team for the opportunity.

 \bibliographystyle{elsarticle-num} 
 \bibliography{cas-refs}





\end{document}